\documentclass[conference,letterpaper,10pt]{IEEEtran}

\usepackage[utf8]{inputenc}
\usepackage[T1]{fontenc}
\usepackage{amsmath, amssymb, amsfonts, mathtools}
\usepackage{graphicx}
\usepackage{multirow, multicol}
\usepackage{nicefrac}
\usepackage{cite}
\usepackage{url}
\usepackage{flushend}
\usepackage{ifthen}
\usepackage{comment}
\usepackage{xcolor}
\usepackage{balance}
\usepackage{amsthm}
\newtheoremstyle{normalstyle}  % Custom upright style
  {}{}                         % Space above and below
  {\normalfont}               % Body font
  {}                          % Indent amount
  {\normalfont}               % Theorem head font
  {.}                         % Punctuation after theorem head
  { }                         % Space after theorem head
  {}                          % Theorem head spec

\theoremstyle{normalstyle}
\newtheorem{theorem}{\textit{Theorem}}

\newtheorem{proposition}{\textit{Proposition}}
\newtheorem{remark}{\textit{Remark}}
\newtheorem{example}{\textit{Example}}
\newtheorem{corollary}{\textit{Corollary}}
\begin{document}
\title{A Generalized Cram\'er–Rao Bound Using Information Geometry} 

% %%% Single author, or several authors with same affiliation:
% \author{%
%  \IEEEauthorblockN{Author 1 and Author 2}
% \IEEEauthorblockA{Department of Statistics and Data Science\\
%                    University 1\\
 %                   City 1\\
  %                  Email: author1@university1.edu}% }

%%% Several authors with up to three affiliations:
\author{%
  \IEEEauthorblockN{Satyajit Dhadumia}
  \IEEEauthorblockA{Department of Mathematics\\
                    Indian Institute of Technology Palakkad\\
                    Kerala, India\\
                    Email: 212114001@smail.iitpkd.ac.in,\\ \hspace{.8cm}satya22dhadumia@gmail.com}
  \and
  \IEEEauthorblockN{M. Ashok Kumar}
  \IEEEauthorblockA{Department of Mathematics\\
                    Indian Institute of Technology Palakkad\\
                    Kerala, India\\
                    Email: ashokm@iitpkd.ac.in}
}

\maketitle

%%%%%%
%% Abstract: 
%% If your paper is eligible for the student paper award, please add
%% the comment "THIS PAPER IS ELIGIBLE FOR THE STUDENT PAPER
%% AWARD." as a first line in the abstract. 
%% For the final version of the accepted paper, please do not forget
%% to remove this comment!
%%

\begin{abstract}
In information geometry, statistical models are considered as differentiable manifolds, where each probability distribution represents a unique point on the manifold. A Riemannian metric can be systematically obtained from a divergence function using Eguchi's theory (1992); the well-known \textit{Fisher-Rao metric} is obtained from the Kullback-Leibler (KL) divergence. The geometric derivation of the classical Cram\'er-Rao Lower Bound (CRLB) by Amari and Nagaoka (2000) is based on this metric. In this paper, we study a Riemannian metric obtained by applying Eguchi's theory to the Basu-Harris-Hjort-Jones (BHHJ) divergence (1998) and derive a generalized Cram\'er–Rao bound using Amari-Nagaoka's approach. There are potential applications for this bound in robust estimation.
\end{abstract}
\section{Introduction}
\label{sec:introduction}
Information geometry is an active area of research that applies differential geometry to statistical models in an effort to address challenging problems in mathematical statistics. The foundations of it can be found in C. R. Rao's pioneering work in 1945 \cite{crrao1945} where he introduced the \textit{Fisher-Rao metric} as a Riemannian metric on the space of probability distributions. The modern development of the field owes much to Shun-ichi Amari, as his significant contributions have influenced the theoretical foundation of information geometry. 
The publication of the foundational book \textit{Methods of Information Geometry} in 2000 by Shun-ichi Amari and Hiroshi Nagaoka \cite{AmariHiroshi2000} marked a significant milestone, offering a comprehensive introduction to the field. Since then, the domain has grown substantially, with numerous books contributing to its expansion, including \cite{arwini2008information,amari2016information}, and \cite{ay2017information}, each addressing various aspects and applications of information geometry. We begin with a brief introduction to information geometry.

Let \( S \)  denote a family of probability distributions \( p_{\boldsymbol{\theta}} \) on a finite set $\mathcal{X}$, parametrized by \( \boldsymbol{\theta} = (\theta_1, \theta_2, \ldots, \theta_k) \in \Theta \), where \( \Theta \) is an open subset of \( \mathbb{R}^k \). We call $S$ a statistical model. In the information geometric framework, \( S \) is regarded as a \( k \)-dimensional manifold or statistical manifold with a coordinate map \( \psi : S \to \mathbb{R}^k \), defined by \( \psi(p_{\boldsymbol{\theta}}) = \boldsymbol{\theta} \). For example, the family of Bernoulli distributions 
$
S= \{Ber(\theta): \theta\in (0,1)\}
$
forms a 1-dimensional statistical manifold. The \textit{Fisher information matrix} at a point \( p_{\boldsymbol{\theta}}\in S \) is a \( k \times k \) matrix \( G (\boldsymbol{\theta})= [g_{ij}(\boldsymbol{\theta})] \), \( 1 \leq i, j \leq k \), with the \((i,j)^{\text{th}}\) entry given by  
\begin{align}
    \label{eqn:metric-component}
   g_{ij}(\boldsymbol{\theta}) = \mathbb{E}_{\boldsymbol{\theta}}\big[ \partial_i \log p_{\boldsymbol{\theta}}(X)\partial_j\log p_{\boldsymbol{\theta}}(X) \big], 
\end{align}
where $\partial_i \coloneqq \frac{\partial}{\partial \theta_i}$ and $\mathbb{E}_{\boldsymbol{\theta}}[\cdot]$ denotes the expectation with respect to $p_{\boldsymbol{\theta}}$. Here $g_{ij}(\boldsymbol{\theta})$ is finite for all $\boldsymbol{\theta}$ and all $i,j$, and $g_{ij}:\Theta \to \mathbb{R}$ is $C^\infty$ smooth \cite{fisher1922mathematical, AmariHiroshi2000}.

The matrix \( G (\boldsymbol{\theta})\) is symmetric and positive semi-definite. When \( G (\boldsymbol{\theta})\) is strictly positive definite symmetric for every point \( p_{\boldsymbol{\theta}} \in S \), it defines a unique Riemannian metric (inner product) on the tangent space \( T_{p_{\boldsymbol{\theta}}}(S) \): $g(\boldsymbol{\theta})\coloneqq \sum_{i,j=1}^k g_{ij}(\boldsymbol{\theta})x^i y^j$ for any two tangent vectors $X,\,Y\in T_{p_{\boldsymbol{\theta}}}(S)$ where $X=\sum_{i=1}^k x^i\partial_i,\,Y=\sum_{j=1}^k y^{j}\partial_j$ with $x^i,\,y^j\in \mathbb{R}$, thereby transforming \( S \) into a Riemannian manifold. This is often referred to as the \textit{Fisher-Rao metric} or \textit{Fisher information metric} \cite{crrao1945}, \cite{AmariHiroshi2000}. This metric can also be derived from the Kullback-Leibler (KL) divergence using Eguchi's theory \cite[Sec.~2]{Eguchi1992} (c.f.~\cite[Th.~2.6]{eguchi2022minimum}), which defines a unique Riemannian metric on a statistical manifold from a divergence function\footnote{A divergence is a non-negative function \(D\) defined on a manifold \(M\) such that \(D(p, q) = 0\) iff \(p = q\).}. Based on this metric, Amari and Nagaoka \cite[Sec.~2.5]{AmariHiroshi2000} re-derived the classical Cram\'er–Rao Lower Bound (CRLB) for finite support statistical models through the information geometric approach. This geometric framework connects statistical quantities with geometric notions.

Let $X_1,X_2, \dots , X_n$ be an i.i.d. sample generated according to a probability distribution $p_{\boldsymbol{\theta}}\in S$. Let $\textbf{X}=(X_1,X_2,\cdots,X_n)$ and $\textbf{x}=(x_1,x_2,\cdots,x_n)$ be a realization of $\textbf{X}$. The maximum likelihood estimate $\hat{\theta}_{MLE}(\textbf{x})$ is the value of the parameter $\boldsymbol{\theta}$ that maximizes the likelihood function $L_n({\boldsymbol{\theta}}; \textbf{x})\coloneqq\Pi_{i=1}^n p_{\boldsymbol{\theta}}(x_i)$, that is, $$\hat{\theta}_{MLE}(\textbf{x}) = \arg\max_{{\boldsymbol{\theta}} \in \Theta} L_n({\boldsymbol{\theta}}; \textbf{x}).$$ 

A necessary condition for the existence of the maximum likelihood estimator (MLE) $\hat{\theta}_{MLE}(\textbf{X})$ is given by 
\begin{align}
    \label{eqn: score equation}
 \sum_{i=1}^{n}s(\boldsymbol{\theta};x_i)=0.
\end{align}
Eq.~\eqref{eqn: score equation} is known as the score equation or the estimating equation, where $s(\boldsymbol{\theta};x)\coloneqq\nabla_{\boldsymbol{\boldsymbol{\theta}}}\log p_{\boldsymbol{\theta}}(x)$.

Let $\hat{p}(\cdot)$ denote the empirical probability distribution of the sample $X_{1},X_{2},\cdots,X_{n}.$
Then \eqref{eqn: score equation} can be re-written as $$\sum_{x}\hat{p}(x)s(\boldsymbol{\theta};x)=0.$$ 

Now let us suppose that  the sample $X_1,X_2, \dots , X_n$ is generated according to a mixture distribution defined by 
\begin{equation}
    \label{eqn:mixture distribution}
    p_{\epsilon,\boldsymbol{\theta}} = (1 - \epsilon)p_{\boldsymbol{\theta}} + \epsilon \delta,
\end{equation}
where \(p_{\boldsymbol{\theta}}\in S\) denotes the true underlying distribution, \(\delta\) denotes the distribution of the outliers and \(\epsilon\in [0,1)\) denotes the proportion of contamination. We assume \(\delta\) has the same support as $p_{\boldsymbol{\theta}}$.

Usual MLE would seek to fit the data to $p_{\epsilon,\boldsymbol{\theta}}$. Hence its estimates are potentially affected by the outliers. Robust estimation focuses on fitting the data to the true underlying distribution $p_{\boldsymbol{\theta}}$, thereby minimizing the influence of the outliers. To achieve this, the estimating equation must be modified to down-weight the effect of the outliers, allowing more reliable estimation.
Basu et al. \cite{BasuHarris1998} proposed the following estimating equation:
\begin{align}
\label{eqn:basu-estimating-eqn}
    \frac{1}{n}\sum_{i=1}^{n}p_{\boldsymbol{\theta}}(x_{i})^{\alpha}s(\boldsymbol{\theta};x_{i}) 
    =\mathbb{E}_{\boldsymbol{\theta}}[p_{\boldsymbol{\theta}}(X)^{\alpha}s(\boldsymbol{\theta};X)], \, \alpha> 0.
\end{align}
 The intuition behind this equation is that if $x$ is an outlier, then ${p_{\boldsymbol{\theta}}(x)}^{\alpha}$ will be small for large values of $\alpha$. \eqref{eqn:basu-estimating-eqn} serves as the first-order optimality condition for the following likelihood function, known as the BHHJ likelihood function  \cite{BasuHarris1998}:
\begin{align}
\label{eqn:basu-likelihood}
     l^{(\alpha)}_{n}&(\boldsymbol{\theta};\textbf{x})\nonumber\\
     &\coloneqq\frac{1}{n}\sum_{i=1}^{n}\frac{(1+\alpha){p_{\boldsymbol{\theta}}(x_{i})}^{\alpha}-1}{\alpha}-\sum_{x\in \mathcal{X}}p_{\boldsymbol{\theta}}(x)^{1+\alpha}.
\end{align}
As $\alpha \to 0,$ $l^{(\alpha)}_{n}(\boldsymbol{\theta};\textbf{x})$ corresponds to 
\begin{align*}
    l_n(\boldsymbol{\theta};\textbf{x})\coloneqq\frac{1}{n}\sum_{i=1}^{n}\log p_{\theta}(x_{i}),
\end{align*}
the usual log-likelihood function. Consequently, it is regarded as a robust alternative to $l_n(\boldsymbol{\theta};\textbf{x})$, with its maximizer serving as a robust estimator, called the BHHJ estimator \cite{BasuHarris1998}.

%This BHHJ estimator shares the asymptotic normality property [3] as the Maximum likelihood estimator, which states that if $\hat{\theta}^{(\alpha)}_{n}\coloneqq\hat{\theta}^{(\alpha)}_{n}(\textbf{X})$ 
 %denotes the BHHJ estimator of $\boldsymbol{\theta}$ then, under some 
   % regularity conditions, $$\sqrt{n}(\hat{\theta}^{(\alpha)}_{n}-\boldsymbol{\theta})\xrightarrow[n \to \infty]{d} \mathcal{N}\bigg(\boldsymbol{0},({I_1^{(\alpha)}})^{-1}\bigg),$$
   % where $$({I_1^{(\alpha)}})^{-1}\coloneqq(G_{1}^{(\alpha)})^{-1}K_{1}^{(\alpha)}(G_{1}^{(\alpha)})^{-1}$$ with 
    %$$G_1^{(\alpha)}\coloneqq\mathbb{E}_{p_{\boldsymbol{\theta}}}\big[p_{\theta}(X)^{\alpha}\nabla_{\boldsymbol{\theta}} \log p_{\boldsymbol{\theta}}(X)\big(\nabla_{\boldsymbol{\theta}} \log p_{\boldsymbol{\theta}}(X)\big)^\top \big],$$ and $$K_1^{(\alpha)}\coloneqqV_{p_{\boldsymbol{\theta}}}\big[p_{\boldsymbol{\theta}}(X)^{\alpha}\nabla_{\boldsymbol{\theta}}\log p_{\boldsymbol{\theta}}(X)\big].$$
It is easy to see that maximizing \eqref{eqn:basu-likelihood} is equivalent to minimizing the following divergence function:
\begin{align}
\label{eqn:basu-divergence}
    &B^{(\alpha)}(\hat{p},p_{\boldsymbol{\theta}})\nonumber\\
    &=\sum_{x\in \mathcal{X}}\Big\{ \frac{{p_{\boldsymbol{\theta}}(x)}^{1+\alpha}}{1+\alpha}-\frac{\hat{p}(x)p_{\boldsymbol{\theta}}(x)^{\alpha}}{\alpha}+\frac{\hat{p}(x)^{1+\alpha}}{\alpha (1+\alpha)}\Big\}.
\end{align}
This is analogous to the fact that maximizing the usual likelihood function is same as minimizing the KL-divergence \cite [Lem.~3.1]{csiszar2004information}. Also, as $\alpha\to 0$, $B^{(\alpha)}(\hat{p},p_{\boldsymbol{\theta}})$ coincides with
\[
I(\hat{p},p_{\boldsymbol{\theta}})\coloneqq\sum_{x\in\mathcal{X}} \hat{p}(x)\log\frac{\hat{p}(x)}{p_{\boldsymbol{\theta}}(x)},
\]
the KL-divergence or relative entropy. $B^{(\alpha)}$ is known as the BHHJ divergence  \cite{BasuHarris1998}.

\begin{remark}
\label{rem:Bregman}
$B^{(\alpha)}$ belongs to the class of Bregman divergences $D_{\varphi}$ with $\varphi(t) = {(t^{1+\alpha}-t})/{\alpha (1+\alpha)}$ \cite[Sec.~2.1]{JonesHijot2001} (c.f.~\cite[Sec.~2]{mukherjee2019b}). 
\end{remark}
Rest of the paper is organized as follows. In Section \ref{sec:alpha-fisher-information}, we focus on the $\alpha$-Fisher information metric, derived from the BHHJ divergence using Eguchi's theory. We will explore the geometry induced by this metric on the space of all probability distributions and introduce a key theorem along with two corollaries. In section \ref{sec:alpha-CRLB}, we present our main result, a generalized Cram\'er–Rao bound linked to the $\alpha$-Fisher information metric. Furthermore, we will discuss its potential application in BHHJ estimation. We end the paper with a summary and concluding remarks in Section \ref{sec:concluding-remarks}. Related works in this direction include \cite{kumar2018information}, \cite{kumar2020cram}, \cite{AshokMishra2021} and \cite{mishra2020generalized}.

\section{The $\alpha$-Fisher information metric and geometry induced on a manifold}
\label{sec:alpha-fisher-information}
Let $\mathcal{P}\coloneqq\mathcal{P}(\mathcal{X})$ denote the space of all probability distributions (strictly positive) on $\mathcal{X}$ (the probability simplex). This can be regarded as a $|\mathcal{X}|-1$ dimensional statistical manifold. Let $ \boldsymbol{\theta} = (\theta_1, \dots, \theta_k)$ be a coordinate system of $\mathcal{P}$. By using Eguchi's theory, we can define a Riemannian metric $g^{(\alpha)}(\boldsymbol{\theta})$, referred to as the $\alpha$-Fisher information metric on $\mathcal{P}$, derived from the BHHJ divergence, characterized by the matrix $G^{(\alpha)}(\boldsymbol{\theta})=[g_{ij}^{(\alpha)}(\boldsymbol{\theta})]$ with
\begin{eqnarray}
\label{eqn:alpha-Fisher-information}
    g^{(\alpha)}_{ij}(\boldsymbol{\theta}) &\coloneqq& -B^{(\alpha)}(\partial_i|\partial_j)\nonumber\\
    &\coloneqq& -\frac{\partial}{\partial \theta_i} \frac{\partial}{\partial {\theta}_j'} B^{(\alpha)}(p_{\boldsymbol{\theta}},p_{\boldsymbol{\theta}'})\big|_{\boldsymbol{\theta}^{'}=\boldsymbol{\theta}}\nonumber\\
         &=& \frac{1}{\alpha}\frac{\partial}{\partial \theta_i} \frac{\partial}{\partial {\theta}_j'} \Big[\sum_{x\in \mathcal{X}}p_{\boldsymbol{\theta}}(x)p_{\boldsymbol{\theta}'}(x)^{\alpha}\Big]\Big|_{\boldsymbol{\theta}^{'}=\boldsymbol{\theta}}\nonumber\\
         &=& \frac{1}{\alpha}\sum_{x\in \mathcal{X}}\Big\{\frac{\partial}{\partial \theta_i} \frac{\partial}{\partial {\theta}_j'}\big[p_{\boldsymbol{\theta}}(x)p_{\boldsymbol{\theta}'}(x)^{\alpha}\big]\Big\}\Big|_{\boldsymbol{\theta}^{'}=\boldsymbol{\theta}}\nonumber\\
         &=& \sum_{x\in \mathcal{X}}\Big[p_{\boldsymbol{\theta}'}(x)^{\alpha-1}\frac{\partial}{\partial \theta_i}p_{\boldsymbol{\theta}}(x)\frac{\partial}{\partial {\theta}_j'}p_{\boldsymbol{\theta}'}(x)\Big]\Big|_{\boldsymbol{\theta}'=\boldsymbol{\theta}}\nonumber\\
         &=& \sum_{x\in \mathcal{X}}\Big[p_{\boldsymbol{\theta}}(x)^{{\alpha}-1}\frac{\partial}{\partial \theta_i} p_{\boldsymbol{\theta}}(x) \frac{\partial}{\partial {\theta}_j}p_{\boldsymbol{\theta}}(x)\Big]\nonumber\\
         &=& \mathbb{E}_{\boldsymbol{\theta}}\big[p_{\boldsymbol{\theta}}(X)^{\alpha}\partial_i\log p_{\boldsymbol{\theta}}(X) \partial_j \log p_{\boldsymbol{\theta}}(X)\big].
       \end{eqnarray}
       
%We know that if $g$ is a Riemannian metric, then $cg,\text{ where } c\in\mathbb{R}$ and $c>0$, is also a Riemannian metric.
%Therefore, we can consider $g^{(\alpha)}(\boldsymbol{\theta})=\frac{1}{1+\alpha}g^{(\alpha)}(\boldsymbol{\theta})$ as a Riemannian metric on $\mathcal{P}$, referred to as the $\alpha$-Fisher information metric, defined by the matrix $G^{(\alpha)}(\boldsymbol{\theta})=[g^{(\alpha)}_{ij}(\boldsymbol{\theta})]$ with $$g^{(\alpha)}_{ij}(\boldsymbol{\theta})=\mathbb{E}_{p_{\boldsymbol\theta}}\big[ \partial_i\log p_{\boldsymbol{\theta}}(X) \cdot p_{\boldsymbol{\theta}}(X)^{\alpha}\partial_j\log p_{\boldsymbol{\theta}}(X)\big
%].$$  %We assume that $g_{ij}^{(\alpha)}(\boldsymbol{\theta})$ is finite for all $\boldsymbol{\theta}$ and all $i,j$, and that $g_{ij}^{(\alpha)}:\Theta \to \mathbb{R}$ is $C^\infty$ smooth.
We will follow the same setting as in \cite[Sec.~2.5]{AmariHiroshi2000}. Notice that $\mathcal{P}\subset \mathbb{R}^{\mathcal{X}}\coloneqq\{A\,|\, A:\mathcal{X}\to \mathbb{R}\}$. The tangent space $T_p(\mathcal{P})$ is identified with the linear space $\mathcal{A}_0
\coloneqq\{A\in \mathbb{R}^{\mathcal{X}}|\,\sum_x A(x)=0\}.$ A tangent vector $X\in T_{p}(\mathcal{P})$ is denoted by $X^{(m)}\in \mathcal{A}_{0}$, called the mixture representation or m-representation of $X$. We write $T_{p}^{(m)}\mathcal{(P)}\coloneqq\{X^{(m)}\,|\,X\in T_{p}(\mathcal{P})\}=\mathcal{A}_{0}.$ Through the mapping $p\mapsto \log p$, a tangent vector $X\in T_{p}(\mathcal{P}) $ can be represented by $X^{(e)}$, defined by $X^{(e)}(x)=X^{(m)}(x)/p(x) \text{ for } x\in \mathcal{X}$, called the exponential or e-representation of $X$. We have $$T_{p}^{(e)}(\mathcal{P})\coloneqq\{X^{(e)}\,|\,X\in T_{p}(\mathcal{P})\}=\{A\in \mathbb{R}^{\mathcal{X}}\,|\,\mathbb{E}_{p}{[A]}=0\}.$$ 

Now consider the embedding $p\mapsto {(p^{\alpha}-1)}/{\alpha}$. Under this, we define 
$X^{(\alpha)}(x)=p(x)^{\alpha}X^{(e)}(x) \text{ for } x\in \mathcal{X} \text{ with } X\in T_{p}(\mathcal{P}),$ $X^{(e)}\in T_{p}^{(e)}(\mathcal{P})$, called $\alpha$-representation of $X$. Let
\[
T_{p}^{(\alpha)}(\mathcal{P}) \coloneqq \big\{X^{(\alpha)} \,|\, X \in T_{p}(\mathcal{P})\big\}.
\]
We have the following proposition.

\begin{proposition}
\begin{align*}
T_{p}^{(\alpha)}(\mathcal{P}) = \big\{A \in \mathbb{R}^{\mathcal{X}}\,|\,\mathbb{E}_{p}\big[p(X)^{-\alpha}A\big] = 0\big\}.    
\end{align*}
\end{proposition}

\begin{proof}
Since $X^{(\alpha)}(x)=p(x)^{\alpha}X^{(e)}(x),$ $p(x)^{-\alpha}X^{(\alpha)}(x)=X^{(e)}(x)$. Hence
\[
\mathbb{E}[p(X)^{-\alpha}X^{(\alpha)}]=\mathbb{E}[X^{(e)}] = 0.
\]
On the other hand, if $A\in \mathbb{R}^{\mathcal{X}}$ such that $\mathbb{E}[p(X)^{-\alpha}A]=0$ then $p^{-\alpha}A\in T_p^{(e)}(\mathcal{P})$. Hence $p^{-\alpha}A=X^{(e)}$ for some $X\in T_{p}(\mathcal{P}).$ This implies $A=p^{\alpha}X^{(e)}$ and hence $A\in T_p^{(\alpha)}(\mathcal{P})$.
\end{proof}

For the natural basis $\partial_i$ of a coordinate system $ \boldsymbol{\theta}$, we have $(\partial_i)_{\boldsymbol{\theta}}^{(m)}=\partial_i p_{\boldsymbol{\theta}}$, $(\partial_i)_{\boldsymbol{\theta}}^{(e)}=\partial_i \log p_{\boldsymbol{\theta}}$,\,$(\partial_i)_{\boldsymbol{\theta}}^{(\alpha)}=p_{\boldsymbol{\theta}}^{\alpha}\partial_i \log p_{\boldsymbol{\theta}}.$\\
With these notations, for any $X,Y\in T_{p}(\mathcal{P})$, we define the $\alpha$-Fisher information metric at a point $p$ on $\mathcal{P}$ as
\begin{align}
\label{eqn:alpha_riemannian_metric}
     g^{(\alpha)}(X,Y)\coloneqq\big\langle X,Y\big\rangle_{p}&\coloneq\sum_x X^{(m)}(x)Y^{(\alpha)}(x)\nonumber \\
     &=\mathbb{E}_{p}\big[X^{(e)}Y^{(\alpha)}\big].
\end{align}
Let $S$ be a submanifold of $\mathcal{P}$ equipped with the metric $g^{(\alpha)}$ as defined in \eqref{eqn:alpha_riemannian_metric}. Let $T_p^*(S)$ be the dual space (cotangent space) of the tangent space $T_p(S)$. For a tangent vector $X\in T_p({S})$, we define the corresponding cotangent vector $\omega_X\in T^*_p{(S)}$ by the mapping $\omega_X:Y\mapsto \big\langle X,Y\big\rangle_p$ for any $Y\in T_p(S),$ where inner product is induced by the $\alpha$-Fisher information metric $g^{(\alpha)}$. This correspondence $X\mapsto\omega_X$ establishes a linear isomorphism between $T_p(S)$ and $T_p^*(S)$. Consequently, the inner product and norm on $T_p^*(S)$ can be defined by 
$\big\langle\omega_X, \omega_Y \big\rangle_p = \big\langle X, Y \big\rangle_p$ and $ \| \omega_X \|_p = \| X \|_p,$
inherited from $T_p(S)$.
Now, for an arbitrary real smooth function $f$ defined on $S$, consider the differential $(df)_p\in T^*_p(S)$ that maps $X$ to $X(f)$. The gradient $(\text{grad} f)_p\in T_p(S)$, corresponding to $(df)_p$, satisfies the following properties:
\begin{align}
\label{eqn:differential 2-form}
    &\big\langle(\text{grad} f)_p,X\rangle_p=X(f)=(df)_p(X), \nonumber \\
    &(\text{grad}f)_p=\sum_{i,j}(g^{ij}(\boldsymbol{\theta}))^{(\alpha)}\partial_j(f)\partial_i,\nonumber\\ 
    &\|(df)_p\|_p^2=\|(\text{grad}f)_p\|^2_p=\sum_{i,j}(g^{ij}(\boldsymbol{\theta}))^{(\alpha)}\partial_j(f)\partial_i(f),
\end{align} where $(g^{ij}(\boldsymbol{\theta}))^{(\alpha)}$ is the $(i,j)^{\text{th}}$ entry of the inverse of $G^{(\alpha)}(\boldsymbol{\theta}).$

The following theorem relates the variance of a random variable $A$, computed with respect to the escort distribution
\begin{equation}
    \label{eqn:escort}
p_{\alpha}(x)\coloneqq \frac{p(x)^{1-\alpha}}{\sum_{y\in\mathcal{X}} p(y)^{1-\alpha}},\quad x\in\mathcal{X},
\end{equation}
to the geometric quantity $\| (d\mathbb{E}[A])_p \|_p^2$, a second-order differential form.
\begin{theorem}
    For $A:\mathcal{X}\to \mathbb{R}$, let $\mathbb{E}[A]$ denote the mapping $p\mapsto \mathbb{E}_{p}[A]$ on $\mathcal{P}$. Then
\begin{align*}
    V_{\alpha}[A]=\frac{1}{\sum_{y}p(y)^{1-\alpha}}\|(d\mathbb{E}[A])_{p}\|_{p}^{2},
\end{align*}
where $V_{\alpha}[A]\coloneqq V_{p_{\alpha}}[A]\coloneq \mathbb{E}_{\alpha}[(A-\mathbb{E}_{\alpha}[A])^2]$ with $\mathbb{E}_{\alpha}[A]\coloneq \mathbb{E}_{p_{\alpha}}[A]$ and the norm $ \| \cdot \|_{p} $ is induced by the $\alpha$-Fisher information metric.
\end{theorem} 
\begin{proof} For every $X\in T_{p}(\mathcal{P})$, we have
\[
X\big(\mathbb{E}[A]\big) = \big\langle\text{grad}\,\mathbb{E}[A], X\big\rangle_{p}.
\]
We observe that
\begin{align*}
     X\big(\mathbb{E}[A]\big)&=X\big(\sum_{x}A(x)p_{\theta}(x)\big)\\
                       &=\sum_{x}X^{(m)}(x)A(x).
\end{align*} 
Since $X^{(e)}(x)=X^{(m)}(x)/p(x) \text{ for } x\in \mathcal{X}$, we can write
$$X\big(\mathbb{E}[A]\big)=\mathbb{E}_{p}\big[ A X^{(e)}\big].$$
Since $\mathbb{E}_{p}[X^{(e)}]=0$, we can write
\begin{align*}
\label{eqn:tgt_vector_to_expectation}
     X\big(\mathbb{E}[A]\big)&=\mathbb{E}_{p}\Big[A X^{(e)}-\frac{\mathbb{E}_{p}[X^{(e)}]}{\sum_{y}p(y)^{1-\alpha}}\mathbb{E}_{p}[p(X)^{-\alpha}A]\Big]\\
                       &= \mathbb{E}_{p}\Big[AX^{(e)}-\frac{X^{(e)}}{\sum_{y}p(y)^{1-\alpha}}\mathbb{E}_{p}[p(X)^{-\alpha}A]\Big]\\
                       &=\mathbb{E}_{p}\Big[X^{(e)}\Big(A-\frac{1}{\sum_{y}p(y)^{1-\alpha}}\mathbb{E}_{p}[p(X)^{-\alpha}A]\Big)\Big].
\end{align*}
If we denote the expectation of $A$ with respect to the escort distribution $p_\alpha$ in \eqref{eqn:escort} by $\mathbb{E}_{\alpha}[A]$, then the above becomes
\begin{equation}
\label{eqn:tgt_vector_to_expectation}
X\big(\mathbb{E}[A]\big) = \mathbb{E}_{p}\big[X^{(e)}\big(A - \mathbb{E}_{\alpha}[A]\big)\big].
\end{equation}
Observe that
 \begin{align}
 \label{eqn:alpha_rep_criterion}
     \mathbb{E}_{p}\big\{&p(X)^{-\alpha}\big(A-\mathbb{E}_{\alpha}[A]\big)\big\}\nonumber\\
     &=\mathbb{E}_{p}\big\{p(X)^{-\alpha}A - p(X)^{-\alpha}\mathbb{E}_{\alpha}[A]\big\}\nonumber\\
     &=\mathbb{E}_{p}\big[p(X)^{-\alpha}A\big] -\mathbb{E}_{p}[p(X)^{-\alpha}]\cdot \mathbb{E}_{\alpha}[A]\nonumber\\
     &=\mathbb{E}_{p}\big[p(X)^{-\alpha}A\big]-\frac{\sum_{x}p(x)^{1-\alpha}}{\sum_{y}p(y)^{1-\alpha}}\mathbb{E}_{p}[p(X)^{-\alpha}A]\nonumber\\
     &=\mathbb{E}_{p}\big[p(X)^{-\alpha}A\big]-\mathbb{E}_{p}\big[p(X)^{-\alpha}A\big]\nonumber\\
     &=0.
  \end{align}   
Thus
\[
A -\mathbb{E}_{\alpha}[A] = Y^{(\alpha)} \in T_{p}^{(\alpha)}(\mathcal{P}),
\]
for some \( Y \in T_{p}(\mathcal{P}) \). Consequently, we can write 
\[
X\big(\mathbb{E}[A]\big) = \mathbb{E}_{p}\big[X^{(e)} Y^{(\alpha)}\big].
\]
Therefore,
\begin{align}
\label{eqn: metric expectation}  
\mathbb{E}_{p}\big[X^{(e)} Y^{(\alpha)}\big] = \big\langle\text{grad}\,\mathbb{E}[A], X\big\rangle_{p}.
\end{align}

Again, from the definition of the $\alpha$-Fisher information metric \eqref{eqn:alpha_riemannian_metric}, we see that 
\begin{equation}
\label{eqn:Ex-Innerproduct}
\mathbb{E}_{p}\big[X^{(e)} Y^{(\alpha)}\big] = \big\langle Y, X\big\rangle_{p}, \quad \text{for any } X, Y \in T_{p}(\mathcal{P}).
\end{equation}
From \eqref{eqn: metric expectation} and \eqref{eqn:Ex-Innerproduct}, we get $Y = \text{grad}\,\mathbb{E}[A] \in T_{p}(\mathcal{P}) $. Hence
\begin{align*}
    \|(d\mathbb{E}[A])_{p}\|_{p}^{2}&= \big\langle\text{grad}\,\mathbb{E}[A], \text{grad}\,\mathbb{E}[A]\big\rangle_{p}\\
     &= \big\langle Y, Y\big\rangle_{p} \\
     &= \mathbb{E}_{p}\big[Y^{(e)} Y^{(\alpha)}\big].
\end{align*}
Since \( A -\mathbb{E}_{\alpha}[A] = Y^{(\alpha)} \),
\begin{equation}
    \label{eqn:star1}
\|(d\mathbb{E}[A])_{p}\|_{p}^{2} = \mathbb{E}_{p}\big[Y^{(e)} \big(A -\mathbb{E}_{\alpha}[A]\big)\big].
\end{equation}
We have
\begin{align}
\label{eqn:star0}
  Y^{(e)}(x)&= p(x)^{-\alpha} Y^{(\alpha)}(x)\nonumber\\
  &= p(x)^{-\alpha} \big(A(x) -\mathbb{E}_{\alpha}[A]\big).
\end{align}
Using \eqref{eqn:star0} in \eqref{eqn:star1}, we get 
\begin{align}
\|(d\mathbb{E}[A])_{p}\|_{p}^{2}
&= \mathbb{E}_{p}\big[p(X)^{-\alpha} \big(A -\mathbb{E}_{\alpha}[A]\big)^2\big]\nonumber \\
&= \sum_{y} p(y)^{1-\alpha} \mathbb{E}_{\alpha} \big[\big(A -\mathbb{E}_{\alpha}[A]\big)^2\big]\nonumber \\& = \sum_{y} p(y)^{1-\alpha} V_{\alpha}[A].\nonumber
\end{align}
Thus
$V_{\alpha}[A] = \frac{1}{\sum_{y} p(y)^{1-\alpha}} \|(d\mathbb{E}[A])_{p}\|_{p}^{2}.
$
\end{proof}
\begin{corollary}
\label{corollarry 1}
 If $S$ is a submanifold of $\mathcal{P}$, then $$V_{\alpha}[A]\geq\frac{1}{\sum_{y}p(y)^{1-\alpha}}\|(d\mathbb{E}[A]|_{S})_{p}\|_{p}^{2},$$ with equality holding if and only if
$$A-\mathbb{E}_{\alpha}[A]\in T_{p}^{(\alpha)}(S).$$
\end{corollary}
\begin{proof} For every real smooth function $f$, $(\text{grad}f|_{S})_p$ is the orthogonal projection of $(\text{grad}f)_p$ onto $T_p(S)$. Utilizing this fact in the previous theorem completes the proof.
\end{proof}
\begin{corollary}
\label{corollarry 2}
    For $B:\mathcal{X}^n\to \mathbb{R}$, let $\mathbb{E}[B]$ denote the mapping $p\mapsto \mathbb{E}_{p}[B]$ on the submanifold $S_n$ of $\mathcal{P}_n\coloneqq\{p:\mathcal{X}^n\to \mathbb{R}\big|\,p(\textbf{x})>0 \text{ for all }\textbf{x}\in \mathcal{X}^n, \, \sum_{\textbf{x}}p(\textbf{x}) = 1\}$. Then $$V_{\alpha}[B]\geq\frac{1}{\sum_{\textbf{y}\in\mathcal{X}^n}p(\textbf{y})^{1-\alpha}}\|(d\mathbb{E}[B]|_{S_n})_{p}\|_{p}^{2},$$ where the variance $V_{\alpha}[\cdot]$ is with respect to the escort distribution
    \[
    p_{\alpha}(\textbf{x})\coloneq \frac{p(\textbf{x})^{1-\alpha}}{\sum_{\textbf{y}}p(\textbf{y})^{1-\alpha}},\quad \textbf{x}\in \mathcal{X}^n,
    \]
    with equality if and only if
\[
B-\mathbb{E}_{\alpha}[B]\in T_{p}^{(\alpha)}(S_n).
\]
\end{corollary}
\begin{proof}
     Follows from \textit{Corollary} \ref{corollarry 1}.
\end{proof}
\section{A Generalized Cram\'er–Rao bound}
\label{sec:alpha-CRLB}
We now state our main result: a generalized Cram\'er–Rao bound for a finite sample.
\begin{theorem}
    Let $X_1, X_2, \dots , X_n$ be an i.i.d. sample generated from a mixture distribution of the form \eqref{eqn:mixture distribution}. Let $\hat{\theta}$ be an unbiased estimator of $\boldsymbol{\theta}$ with respect to $p_{\boldsymbol{\theta}}$.
Then we have
\begin{equation}
\label{eqn:alpha-CRLB}
    V_{\alpha,\boldsymbol{\theta}}[\hat{\theta}]\geq \frac{1}{\sum_{\textbf{y}}p_{\boldsymbol{\theta}}(\textbf{y})^{1-\alpha}}[G_n^{(\alpha)}(\boldsymbol{\theta})]^{-1},
\end{equation}
where $G_n^{(\alpha)}(\boldsymbol{\theta}) = [g^{(\alpha)}_{ij}(\boldsymbol{\theta})]$ with $$g^{(\alpha)}_{ij}(\boldsymbol{\theta}) = \mathbb{E}_{\boldsymbol\theta}\big[ p_{\boldsymbol{\theta}}(\textbf{X})^{\alpha}\partial_i\log p_{\boldsymbol{\theta}}(\textbf{X})\partial_j\log p_{\boldsymbol{\theta}}(\textbf{X})\big]$$ and $V_{\alpha,\boldsymbol{\theta}}[\hat{\theta}]=[\text{Cov}_{\alpha,\boldsymbol{\theta}}\big(\hat{\theta}_i(\textbf{X}),\hat{\theta}_j(\textbf{X}))\big]$ is the covariance matrix and the covariance is with respect to the escort distribution $$p_{\alpha,{\boldsymbol{\theta}}}{(\textbf{x})}\coloneqq\frac{p_{{\boldsymbol{\theta}}}{(\textbf{x})}^{1-\alpha}}{\sum_{\textbf{y}}p_{{\boldsymbol{\theta}}}{(\textbf{y})}^{1-\alpha}},\quad \textbf{x}\in\mathcal{X}^n.$$
(Note: For two matrices $A_1$ and $A_2$, we write $A_1\geq A_2$ to indicate that the matrix $A_1-A_2$ is positive semi-definite.)
\end{theorem}
\begin{proof} 
Let $B = c^\top \hat{\theta}, $ where $c\in\mathbb{R}^k$. From \eqref{eqn:differential 2-form}, we obtain
\[
\|(d\mathbb{E}[B]_p)\|_p^2=c^\top [G_n^{(\alpha)}(\boldsymbol{\theta})]^{-1}c.
\]
Applying \textit{Corollary} \ref{corollarry 2} to $S_n$, we obtain
\begin{equation*}
    c^\top V_{\alpha,{\boldsymbol{\theta}}}[\hat{\theta}]c\geq c^\top \frac{1}{\sum_{\textbf{y}}p_{\boldsymbol{\theta}}(\textbf{y})^{1-\alpha}}[G_n^{(\alpha)}(\boldsymbol{\theta})]^{-1}c.
\end{equation*}
This proves the theorem.
\end{proof}
\begin{remark}
    For \( \alpha = 0 \), \( G_n^{(\alpha)}(\boldsymbol{\theta}) \) reduces to the Fisher information matrix \( G_n(\boldsymbol{\theta})=n G(\boldsymbol{\theta})\), and in this case, bound \eqref{eqn:alpha-CRLB} matches the classical CRLB.
\end{remark}
We know that the sample mean is an unbiased estimator that achieves the classical CRLB. The example below shows that the sample mean satisfies the generalized CRLB \eqref{eqn:alpha-CRLB} with strict inequality even in the `no contamination' case ($\epsilon=0$).
\begin{example}
   Let \( X_1,X_2, \dots , X_n\), $n>1$ be an i.i.d. sample from a Bernoulli distribution with parameter \( \theta \in (0,1) \), where \( p_{\theta}(x) = \theta^x (1-\theta)^{1-x},\,x\in\{0,1\}.\) Let \( \hat{\theta} =\overline{X}= \frac{1}{n} \sum_{i=1}^n X_i \) with $\mathbb{E}_{\theta}[\overline{X}]=\theta$. One can show that
\[
V_{\alpha,\theta}[\overline{X}] > \frac{1}{\sum_{\mathbf{y}} p_{\theta}(\mathbf{y})^{1-\alpha}} [G_n^{(\alpha)}(\theta)]^{-1}, \text{ for } \alpha \neq 0.
\]
Note that for \( \alpha = 0 \), the above strict inequality reduces to equality, which corresponds to the classical CRLB.
\end{example}

We now highlight a notable result: under the true model, the covariance of the estimator that attains the generalized CRLB \eqref{eqn:alpha-CRLB} achieves the asymptotic covariance of the BHHJ estimator.
\begin{proposition}
\label{prop:Basu_assym_variance}
    If equality holds in \eqref{eqn:alpha-CRLB}, then
\begin{align}
\label{eqn:equality}
    V_{\boldsymbol{\theta}}[\hat{\theta}] = [{I^{(\alpha)}_n}(\boldsymbol{\theta})]^{-1},
\end{align} where 
\[
I^{(\alpha)}_n(\boldsymbol{\theta}) = G_n^{(\alpha)}(\boldsymbol{\theta})[K_n^{(\alpha)}(\boldsymbol{\theta})]^{-1} G_{n}^{(\alpha)}(\boldsymbol{\theta})
\]
is the inverse of the asymptotic covariance matrix of the BHHJ estimator \cite{basu2011statistical}, \cite{BasuHarris1998}, with
\begin{align*}
    &G_n^{(\alpha)}(\boldsymbol{\theta})\coloneqq\mathbb{E}_{\boldsymbol{\theta}}[p_{\boldsymbol{\theta}}(\textbf{X})^{\alpha}\nabla_{\boldsymbol{\theta}} \log p_{\boldsymbol{\theta}}(\textbf{X})\big(\nabla_{\boldsymbol{\theta}} \log p_{\boldsymbol{\theta}}(\textbf{X})\big)^\top],\\
    &K_n^{(\alpha)}(\boldsymbol{\theta})\coloneqq V_{\boldsymbol{\theta}}\big[p_{\boldsymbol{\theta}}(\textbf{X})^{\alpha}\nabla_{\boldsymbol{\theta}}\log p_{\boldsymbol{\theta}}(\textbf{X})\big].
\end{align*}
\end{proposition}
\begin{proof}
    Assume that the equality holds in \eqref{eqn:alpha-CRLB}. From \textit{Corollary} \ref{corollarry 2}, we have
$$\hat{\theta}-\mathbb{E}_{\alpha,\boldsymbol{\theta}}[\hat{\theta}]\in T_{p}^{(\alpha)}(S_n).$$
%here, $$\Tilde{S}\coloneqq\{p_{\boldsymbol{\theta}}(\textbf{x})|\textbf{x}\in \mathcal{X}^n,\boldsymbol{\theta}\in \Theta\}.$$ 
That is,
\begin{align}
    \label{eqn:star2}
        \hat{\theta}-\mathbb{E}_{\alpha,\boldsymbol{\theta}}[\hat{\theta}] = m(\boldsymbol{\theta})p_{\boldsymbol{\theta}}(\textbf{x})^{\alpha}\nabla_{\boldsymbol{\theta}}\log p_{\boldsymbol{\theta}}(\textbf{x})
\end{align}
for some $k\times k$ matrix $m(\boldsymbol{\theta})$.
Now we will find $m(\boldsymbol{\theta})$. To do this, observe that \eqref{eqn:alpha-CRLB} with equality can be written as 
\begin{align*}
    \sum_{\textbf{y}}p_{\boldsymbol\theta}(\textbf{y})^{1-\alpha}\mathbb{E}_{\alpha,\boldsymbol{\theta}}\big[\big(\hat{\theta}-\mathbb{E}_{\alpha,\boldsymbol{\theta}}[\hat{\theta}]\big)&\big(\hat{\theta}-\mathbb{E}_{\alpha,\boldsymbol{\theta}}[\hat{\theta}]\big)^{T}\big]\\
    & = [G_n^{(\alpha)}(\boldsymbol{\theta})]^{-1}.
\end{align*}
That is,
\begin{align}
\label{eqn:star3}
\mathbb{E}_{\boldsymbol{\theta}}\big[
p_{\boldsymbol{\theta}}(\mathbf{X})^{-\alpha}
\big(\hat{\theta} - \mathbb{E}_{\alpha, \boldsymbol{\theta}}[\hat{\theta}]\big)
\big(\hat{\theta} - \mathbb{E}_{\alpha, \boldsymbol{\theta}}[\hat{\theta}]\big)^\top
\big] 
= [G_n^{(\alpha)}(\boldsymbol{\theta})]^{-1}.
\end{align}
Substituting \eqref{eqn:star2} in \eqref{eqn:star3} yields
\begin{align*}
m(\boldsymbol{\theta})\mathbb{E}_{\boldsymbol{\theta}}\big[p_{\boldsymbol{\theta}}(\textbf{X})^{\alpha}\nabla_{\boldsymbol{\theta}}\log p_{\boldsymbol{\theta}}(\textbf{X})(\nabla_{\boldsymbol{\theta}}&\log \,p_{\boldsymbol{\theta}}(\textbf{X}))^\top]m(\boldsymbol{\theta})^\top\\
&=[G_n^{(\alpha)}(\boldsymbol{\theta})]^{-1}.
\end{align*}
This implies that $m(\boldsymbol{\theta}) = [G_n^{(\alpha)}(\boldsymbol{\theta})]^{-1}.$
Hence from \eqref{eqn:star2}, we have
    \begin{align}
        \label{eqn:star5}
  \hat{\theta}-\mathbb{E}_{\alpha,\boldsymbol{\theta}}[\hat{\theta}]=[G_n^{(\alpha)}(\boldsymbol{\theta})]^{-1} p_{\boldsymbol{\theta}}
(\textbf{x})^{\alpha}\nabla_{\boldsymbol{\theta}}\log p_{\boldsymbol{\theta}}(\textbf{x}).
      \end{align}
Taking covariance on both sides of \eqref{eqn:star5}, we get
     \begin{align*}
        &V_{\boldsymbol{\theta}}[\hat{\theta}]\\
        &=[G_n^{(\alpha)}(\boldsymbol{\theta})]^{-1}  V_{\boldsymbol{\theta}}[p_{\boldsymbol{\theta}}
(\textbf{X})^{\alpha}\nabla_{\boldsymbol{\theta}}\log p_{\boldsymbol{\theta}}(\textbf{X})][(G_n^{(\alpha)}(\boldsymbol{\theta}))^{-1}]^\top\\
        &=[G_n^{(\alpha)}(\boldsymbol{\theta})]^{-1}  V_{\boldsymbol{\theta}}[p_{\boldsymbol{\theta}}
(\textbf{X})^{\alpha}\nabla_{\boldsymbol{\theta}}\log p_{\boldsymbol{\theta}}(\textbf{X})][G_n^{(\alpha)}(\boldsymbol{\theta})]^{-1},
    \end{align*}
    since $G_n^{(\alpha)}(\boldsymbol{\theta})$ is symmetric.
    Thus
    \begin{align*}
        V_{\boldsymbol{\theta}}[\hat{\theta}] &= [G_n^{(\alpha)}(\boldsymbol{\theta})]^{-1} K_n^{(\alpha)}(\boldsymbol{\theta})[G_n^{(\alpha)}(\boldsymbol{\theta})]^{-1}\\
        &=[{I^{(\alpha)}_n}(\boldsymbol{\theta})]^{-1}.
    \end{align*}
\end{proof}
%Now consider an i.i.d. random sample \( X_1 = x_1, X_2 = x_2, X_3=x_3,X_4=x_4,X_5=x_5\) generated by a mixture distribution of the form $p_{\epsilon, \theta} = (1 - \epsilon)p_\theta + \epsilon p_{\theta'}$, $\epsilon\in [0,1),$
%where \( p_\theta \) and \( p_{\theta'} \) are Bernoulli distributions with parameters \( \theta = \frac{1}{3} \) and \( \theta' = \frac{3}{4} \), respectively.\\
%Note that the sample mean \( \hat{\theta} = \bar{X} \) is neither a robust estimator nor an unbiased estimator of \( \theta \) under this model. Also $V_{p_{\theta}}[\bar{X}]\approx0.0444$ and $({I^{(\alpha)}_n}(\boldsymbol{\theta}))^{-1}\approx 0.0668$ at $\alpha=1$. Thus $V_{p_{\theta}}[\bar{X}]<({I^{(\alpha)}_n}(\boldsymbol{\theta}))^{-1}.$
%\begin{remark}
 %   The procedure utilized here can be extended to the class of Bregman divergences.
%\end{remark}
\section{Summary and Concluding Remarks}
\label{sec:concluding-remarks}
 In this paper, we established a generalized Cram\'er-Rao bound. We first applied Eguchi's theory to the BHHJ divergence and established the $\alpha$-Fisher information metric. This serves as the Riemannian metric on a statistical manifold. We then used the Amari-Nagaoka theory to derive the generalized Cram\'er-Rao bound \eqref{eqn:alpha-CRLB}. The classical CRLB sets a benchmark for the variance of estimators within the maximum likelihood estimator (MLE) framework. If an unbiased estimator satisfies the classical CRLB, it is considered as ``best'' as its variance matches the asymptotic variance of the MLE \cite{Lehman1983}. However, in the presence of contamination, the MLE often exhibits significant bias from the true parameters. To address this issue, various robust estimators have been developed in the literature \cite{simpson1987minimum,windham1995robustifying,BasuHarris1998,jones2001comparison}. In \cite{BasuHarris1998}, Basu et al. proposed a robust procedure to estimate the parameters of a statistical model, called the BHHJ estimator. The BHHJ estimator depends on a parameter $\alpha$. This $\alpha$ offers a balance between robustness and efficiency. They also proved an asymptotic normality result establishing that $[I_n^{(\alpha)}(\theta)]^{-1}$ is the asymptotic variance. In the robust setting, our proposed bound \eqref{eqn:alpha-CRLB} serves as a benchmark for the variance of estimators within the BHHJ estimator framework. In particular, an estimator that is unbiased with respect to the true model and satisfies bound \eqref{eqn:alpha-CRLB} is considered as ``best'' as its variance under the true model matches the asymptotic variance of the BHHJ estimator (\textit{Proposition} \ref{prop:Basu_assym_variance}). Although we restricted to the case $\alpha>0$ to ensure robustness, it is worth noting that the theoretical framework developed in this paper remains valid even when $\alpha < 0$ (except for $\alpha = -1$). $B_{\alpha}$ still qualifies as a Bregman divergence derived from the same generating function described in \textit{Remark} \ref{rem:Bregman}. Thus, while the focus on positive $\alpha$ is motivated by practical considerations, the underlying mathematical structure applies to other $\alpha$'s as well (except $\alpha = 0$ and $\alpha = -1$). The main technical novelty of the paper lies in our ability to write $X(\mathbb{E}[A])$ as in \eqref{eqn:tgt_vector_to_expectation} with $A-\mathbb{E}_{\alpha}[A]$ having the property of an $\alpha$-representation of a tangent vector \eqref{eqn:alpha_rep_criterion}. This also complies with the Riemannian metric \eqref{eqn:alpha_riemannian_metric}. That is, the Riemannian metric is the expected value of the product of the $e$-representation and $\alpha$-representation of two tangent vectors. This framework can be extended to a general Bregman divegence with ${1}/{[p(X)\varphi''(p(X))]}$ replacing $p(X)^{-\alpha}$ in \eqref{eqn:alpha_rep_criterion}. The Riemannian metric in this case will be the same as \eqref{eqn:alpha-Fisher-information} with $p(X)\varphi''(p(X))$ in place of $p(X)^{\alpha}$. This, and the applicability of this framework to other divergence functions, will be the focus of the authors' forthcoming work. In the case of classical CRLB (KL-divergence), it was just $A-\mathbb{E}_{p}[A]$ which can be written as an $e$-representation of a tangent vector and the Riemannian metric is the expected value of the product of the $e$-representations of two tangent vectors. It should also be noted that our bound~\eqref{eqn:alpha-CRLB} currently applies to models with finite support. As part of future work, we aim to extend this result to models with continuous support which will increase the applicability of the bound \eqref{eqn:alpha-CRLB}.
\section{Acknowledgements}
We thank the reviewers for their valuable comments that helped improve the manuscript. Satyajit Dhadumia is supported by a Council of Scientific \& Industrial Research (CSIR) fellowship of the Department of Scientific \& Industrial Research, Ministry of Science and Technology, Government of India. The work was also supported in part by the Science \& Engineering Research Board (SERB) of the Department of Science and Technology, India, MATRICS grant (No. MTR/2022/000922).

   \bibliographystyle{plain}
    \bibliography{references}

\end{document}